\documentclass{siamltex}
\usepackage{amsmath,amssymb,color}
\usepackage{graphicx}
\usepackage{wasysym}
\usepackage{cite}
\numberwithin{equation}{section}
\usepackage{tikz}\usetikzlibrary{arrows}

\graphicspath{{./pictures/}}

\newcommand{\qed}{\hfill$\blacksquare$}

\newcommand{\R}{\mathbb{R}} 
\newcommand{\Z}{\mathbb{Z}}

\newtheorem{thm}{Theorem}[section]
\newtheorem{lem}[thm]{Lemma}
\newtheorem{prop}[thm]{Proposition}

\newtheorem{define}[thm]{Definition}
\newtheorem{notation}[thm]{Notation}

\newtheorem{example}[thm]{Example}
\newtheorem{remark}[thm]{Remark}

\renewcommand{\b}{\beta}
\newcommand{\g}{\gamma}

\newcommand{\e}{\epsilon}

\newcommand{\av}[1]{\left|{#1}\right|}

\newcommand{\norm}[1]{\left\|{#1}\right\|}
\renewcommand{\l}{\left}

\renewcommand{\r}{\right}

\newcommand{\be}{\begin{enumerate}}
\newcommand{\bi}{\begin{itemize}}
\newcommand{\ee}{\end{enumerate}}
\newcommand{\ei}{\end{itemize}}

\newcommand{\G}{{\Gamma}}

\renewcommand{\phi}{\varphi}

\newcommand{\nm}{{n}_{{{{-}}}}}
\newcommand{\no}{{n}_{{{{0}}}}}
\newcommand{\np}{{n}_{{{{+}}}}}

\renewcommand{\L}{\mathcal{L}}
\newcommand{\M}{\mathcal{M}}

\newcommand{\ST}{{\mathcal{ST}}}

\newcommand{\mc}[1]{\mathcal{{#1}}}

\title{The spectral index of signed Laplacians and their structural
  stability}  \author{Jared C. Bronski, Lee DeVille,  Paulina
  Koutsaki\footnote{Department of Mathematics, University of Illinois,
    1409 W. Green St, Urbana IL 61801 USA}}

\begin{document}

\maketitle

\begin{abstract}
  Given a graph Laplacian with positively and negatively weighted
  edges we are interested in characterizing the set of weights that
  give a particular spectral index, i.e.~give a prescribed number of
  positive, zero, and negative eigenvalues.  One of the main results
  of this paper is that the set of signed Laplacians that exhibit
  multiple zero eigenvalues is ``small'', and that eigenvalue
  crossings are nongeneric --- specifically, eigenvalues repel each
  other near zero in a sense that can be made precise.  We exhibit an
  algebraic discriminant that measures the level of repulsion, and
  show that this discriminant admits a combinatorial interpretation.
  Conversely, we exhibit a constructive method for finding the sets of
  Laplacians that exhibit a large degree of degeneracy (many
  eigenvalues at or near zero) in terms of these discriminants.



\end{abstract}

\section{Introduction}

Let $\G = \{\gamma_{ij}\}$ be a symmetric weighted graph, let $\L(\G)$
be its graph Laplacian, and let $n_+(\G)$ be the number of positive
eigenvalues of this Laplacian.  It is a classical result that if
$\gamma_{ij} >0$, then $n_+(\G) = 0$.  In~\cite{Bronski.DeVille.14},
the authors computed bounds for $n_+(\G)$ when $\gamma_{ij}$ are
allowed to be both positive and negative.  We showed that the sign
topology of the graph (i.e., only knowing which edges were positive
and which were negative) determines strict upper and lower bounds on
$n_+(\G)$, and, moreover, that these bounds are saturated: there is
some choice of weights that achieves these bounds.

There are graphs for which $n_+(\G)$ is independent of the magnitude
of all of the weights (these graphs were termed ``rigid''
in~\cite{Bronski.DeVille.14}); however, for `` most'' graphs,
$n_+(\G)$ depends on the magnitudes of the weights attached to each
edge.  A natural question is then, for any particular sign topology,
to describe the set of weights which give a prescribed number of
positive eigenvalues.  In this paper, we study this question in its
most general formulation, describing these sets for any particular
sign topology, as well as considering questions of genericity.

We also consider applications of this question to designing networks
with particular dynamical properties, in particular those that can
differentiate multiple signals, and those that are stable but can
support multiple solutions on long timescales.

\subsection{Results of paper}

If $\G = \{\gamma_{ij}\}$ be a weighted symmetric graph, then we
define $\L(\G)$ to be its graph Laplacian, i.e. $\L(\G)$ is the matrix
whose entries are
\begin{equation}\label{eq:Laplacian}
  \L(\G)_{ij} = \begin{cases} \gamma_{ij},& i\neq j,\\ -\sum_{k\neq i}\gamma_{ik},& i=j.\end{cases}
\end{equation}
We define $(n_-(\G), n_0(\G), n_+(\G))$ to be the number of negative,
zero, and positive eigenvalues of $\L(\G)$.

Let $G = (V,E)$ be a symmetric unweighted graph, where all of the
edges are colored black or red.  We term this colored structure the
``topology'' of the graph $G$.  We also define the two subgraphs $G_+$
(resp.~$G_-$) to be the subgraphs where we consider only black
(resp.~red) edges.  We define $c(\cdot)$ of any graph to be the number
of its connected components.

Let us now choose the convention that black edges will correspond to
positive weights, and red edges to negative weights\footnote{We are
  stealing this convention from the accounting industry.}.  It is
clear that, given a weighted graph $\G$, there is a unique colored
graph $G$ associated with it --- we simply forget the magnitude of the
weights and retain only their signs.  Conversely, for any signed graph
$G$, there is a natural map from the positive orthant to the set of
weighted graphs $\G$.

We then show a number of results about the index of all weighted
graphs with a given topology.  We first state the following results,
which are a generalization and slight restatement of the authors'
previous results in~\cite{Bronski.DeVille.14}:

\begin{enumerate}

\item If we fix the positive weights of the graph and choose the
  negative weights in a neighborhood of zero, then generically the
  Laplacian graph has $c(G_+)-1$ positive eigenvalues, a single zero
  eigenvalue, and the rest positive.  In particular, if $G_+$ is
  connected, then for negative weights sufficiently small, the
  Laplacian is negative semidefinite.

\item If we fix the positive weights of the graph and choose the
  negative weights in a neighborhood of $\infty$, then the Laplacian
  graph has $n - c(G_-)$ positive eigenvalues, a single zero
  eigenvalue, and the remainder negative.  In particular, if there are
  any red edges, then we can destabilize the matrix by choosing their
  weights large enough.
\end{enumerate}

If we consider a one-parameter family of Laplacians with fixed black
weights and move the red weights from zero to infinity along some ray,
then it follows from the above that $\tau:= n - c(G_+) - c(G_-) + 1$
eigenvalues move through zero from left to right.  The next question
is: does this one-parameter family have $\tau$ distinct crossings of
individual eigenvalues, or do multiple eigenvalues cross at the same
time?  (Note here that the Laplacian will always have an eigenvalue
fixed at $0$, so a simple crossing corresponds to an eigenvalue with
multiplicity exactly two.)

To state the answer precisely, we need some notation.  Let $G$ be a
fixed graph with $B$ black edges and $R$ red edges, and denote the
weights on these edges by $w_+\in \R^B$ and $w_-\in \R^R$.  This
induces a map from $\R^B\times \R^R$ to all weighted graphs with a
given topology, by setting the positive weights to be the values of
$w_+$ and the negative weights to be $-w_-$.  If we restrict $w_+,
w_-$ to the positive orthant, then there will be positive weights on
the black edges and negative weights on the red edges.  The mapping
from $\R^B\times \R^R$ to weighted graphs induces a topology on
weighted graphs, and it is in this sense that we use the term generic
below.

The one-parameter family of graphs described above corresponds to
fixed $w_+$ and the ray $tw_-$, $t\in[0,\infty)$ --- this gives a
one-parameter family of graphs, which gives a one-parameter family of
sets of eigenvalues, indexed by $t$.

The main results of this paper are as follows.  We will only state and
prove the results in the case where $\G$ is connected (if not, the
Laplacian is the direct sum of the connected copies, so these results
generalize in a straightforward manner).

\begin{enumerate}

\item Fix any $w_+$ for the positive weights.  For a generic set of
  $w_-$, all of the eigenvalues have distinct crossings, i.e. there
  are $\tau$ distinct values of $t$ at which the Laplacian has a
  single eigenvalue crossing zero.

\item There exists a generic set of $w_+$ such that for any $w_-$, all
  of the eigenvalues have distinct crossings.  Moreover, for any fixed
  $w_+$, if we consider the set of all $w_-$ with $\norm{w_-} = 1$,
  then there is a ``minimum distance'' in $t$ between successive
  eigenvalue crossings.  This shows the phenomenon of level repulsion
  for generic Laplacians.

\item To obtain multiple eigenvalues crossing simultaneously, then
  from the previous two statements, one must be ``doubly nongeneric''
  in a certain sense.  We introduce an algebraic expression, related
  to the discriminant of a polynomial, that determines when this
  nongeneric situation can occur. This discriminant can be expressed
  in terms of the homology of the graph, and corresponds to a signed
  count of certain spanning 2-forests in the graph.

\item Finally, we give sufficient $\ell^1$ conditions on the vector
  $w_-$ that guarantee stability of the Laplacian, and moreover
  characterize these conditions in a combinatorial manner.

\end{enumerate}

There is an obvious duality in the statements above, since the
spectrum of $-L$ is just the spectrum of $L$ times $-1$, so we could
also choose to fix the negative weights $w_-$ and vary $w_+$ and
obtain obvious analogues of the above results.

\subsection{Applications}

We present two applications motivating these questions:

\noindent 1] Consider the problem of designing a linear network to
differentiate multiple signals, where the network structure is
prescribed. More specifically, consider the system
\begin{equation}\label{eq:forcedlinear}
  \frac{dx }{dt} = {A}  x +  f(t), \quad  x\in \R^n, \quad f\colon \R\to\R^n
\end{equation}
where $f(t)$ represents an external signal, the matrix ${A}$
represents the network connectivity, and the vector $x(t)$ represents
the network response. In order for the system to be stable we require
that the spectrum of the matrix ${A}$ lie in the left half-plane. It
is clear that eigenvectors corresponding to eigenvalues far from zero
are difficult to excite, so the main response from such a system is
from the eigenvalues near zero. For such modes the system above acts
as an integrator, integrating the projection of the external signal
onto the eigenmodes.  Being able to recognize multiple signals is the
same as saying that we would like to choose ${A}$ to have a large
dimensional kernel --- one dimension for each signal we would like to
be able to recognize --- or, perhaps, to choose $A$ so that it has
many eigenvalues near zero so that the responses track the signal with
a slow decay.  

Of course, it is not difficult to design a linear system with spectrum
wherever we would like: simply choose the eigenvalues, then any matrix
similar to the corresponding diagonal matrix would do.  However,
notice that in general this gives a dense matrix, and naively it is
not clear how one can choose the eigenstructure so that the eventual
linear system is compatible with a desired topological structure.
However, using the results of this paper, we show how this can be
accomplished: if we can design a network with a $k$-fold degeneracy at
zero, then an open set of perturbations of the weights of such a
system will give the desired network.

\noindent 2] Given a graph $G = (V,E)$ and symmetric coupling
functions $\varphi_{ij}(\cdot) = \varphi_{ji}(\cdot)$, define
\begin{equation}\label{eq:network}
  \frac{d}{dt} x_i = F_i({\mathbf x}) :=  \omega_i + \sum_{(i,j)\in E}\varphi_{ij}(x_j-x_i).
\end{equation}
One famous case of this model is the Kuramoto oscillator
network~\cite{K, Kuramoto.91, S, Acebron.etal.05}, where
$\varphi_{ij}(\cdot) = \gamma_{ij} \sin(\cdot)$.  Assume ${\mathbf
  x}^*$ is a fixed point for~\eqref{eq:network}, i.e. $F_i(\mathbf
x^*) = 0$ for all $i$.  The stability of this point is determined the
index of the Jacobian $J$, where
\begin{equation*}
  J_{ij} = \begin{cases} \varphi_{ij}'(x^*_j-x^*_i),& i\neq j,\\
    -\sum_{k} \varphi_{ik}'(x^*_k-x^*_i), & i = j.\end{cases}
\end{equation*}
In the Kuramoto case, the off-diagonal terms are given by
$\gamma_{ij}\cos(x_j^*-x_i^*)$.  The Jacobian $J$ is a graph Laplacian
of the form~\eqref{eq:Laplacian}; thus, determining the stability
indices for fixed points of~\eqref{eq:network} is related to the
problem studied here~\cite{MS2, Mirollo.Strogatz.2005, MS1,VO1,
  Bronski.DeVille.Park.2012}.  Of course, identifying those fixed
points whose Jacobian is negative semi-definite gives the attracting
fixed points for the system, but in fact being able to determine which
of these points have one unstable eigenvalue is important to
understand metastable transitions for stochastic versions of this
system~\cite{NL-DV-2012}.

If the components of $x^*$ are close enough to each other, then
$\cos(x_j^*-x_i^*) >0$, and the Jacobian is negative semidefinite by
the classical theory (see Theorem 3.1
of~\cite{Ermentrout.92}). However, one might ask about the stability
of ``splay states'', i.e. those stationary points where some of the
components are far enough from each other to make the $\cos(\cdot)$
term negative. For a generic choice of $\omega_i$,
$\varphi'_{ij}(x^*_i-x^*_j)$ is non-zero for all $(i,j)\in E$,
implying that the graph determining $J$ and the graph defined by the
original interactions in~\eqref{eq:network} have the same underlying
topology. Thus we have a fixed network topology, and want to
understand the effect of some edge weights being negative.  The
boundary of the region where the matrix is negative-semidefinite with
a one-dimensional kernel is, of course, the set of points where the
matrix is negative semi-definite with a higher dimensional
kernel~\cite{Bronski.DeVille.Park.2012}, and this provides yet another
motivation for studying this problem.

\section{Statement of main results}

In this section, we present the main results of this paper, leaving
the proofs for later sections.  Many of the definitions in this
section are identical to those of~\cite{Bronski.DeVille.14}, but we
include them here for completeness.

\subsection{Weighted graphs, signed graphs, and the Laplacian}
\begin{define}
\begin{itemize}

\item   A {\bf graph} $G = (V,E)$ is a set $V$ of vertices and a set
  $E\subseteq V\times V$ of edges.  

\item A {\bf signed graph} is the triple $G=(V,E,\sigma)$ where
  $(V,E)$ is a standard graph, and with a map $\sigma\colon
  E\to\{\mbox{red},\mbox{black}\}$.

\item A {\bf weighted graph} is the pair $\G = (V,\{\gamma_{ij}\})$
  where $\gamma_{ij}\in\R$.  The {\bf edges} of $\G$ are those $(i,j)$
  with $\gamma_{ij}\neq 0$, and we say that $\gamma_{ij}$ is the {\bf
    weight} of edge $i\leftrightarrow j$. 
\end{itemize}
\end{define}

Any weighted graph corresponds to a signed graph in an obvious manner;
we will colloquially call this the ``topology'' of the weighted graph.

\begin{define}
  Given a weighted graph $\G$, the {\bf Laplacian} of $\G$ is the
  matrix $\L(\G)$ with
\begin{equation}\label{eq:defofL}
  \L_{ij} = \begin{cases} \quad \gamma_{ij},&i\neq j,\\ -\sum_{k\neq i}\gamma_{ik},& i=j.\end{cases}
\end{equation}
The {\bf index} of (the Laplacian of) $\G$ is the triple of integers
\begin{equation}\label{eq:defofindex}
  (\nm(\G), \no(\G), \np(\G))
\end{equation}
giving the number of negative, zero, and positive eigenvalues of
$\L(\G)$.
\end{define}

\begin{notation}
  If $G$ is a signed graph, we will denote by $G_+$ the subgraph
  containing only the plus edges, and $G_-$ the subgraph containing
  only the minus edges, with a similar convention for $\Gamma_\pm$
  when considering weighted graphs.  We also denote $c(G)$ as the
  number of connected components of a graph, so that $c(G_+)$ is the
  number of connected components of a signed graph when we only
  consider positive edges, etc.  We only consider connected graphs in
  this paper, so that $c(G) = 1$, but we allow for $c(G_+)$ and
  $c(G_-)$ to be larger than one.  If $\Gamma$ is a weighted graph,
  then we can associate it to a signed graph in the obvious way, and
  thus all of the notions above make sense as well, i.e $\Gamma_+$,
  $c(\Gamma_+)$, etc.
\end{notation}

\subsection{Crossing polynomial and spectral variety}

\begin{define}\label{def:M}
If ${T}$ is a tree, define $\pi({T})$ to be the product over the edge
weights in the tree
\begin{equation}\label{eq:defofpi}
  \pi({T}) := \prod_{i<j,(i,j)\in E(T)} \gamma_{ij}.
\end{equation}
Let $\G$ be a weighted graph with $\av{V(\G)}=N$, and $\L(\G)$ be its
graph Laplacian.  We know that $\L(\G)$ has a zero eigenvalue, and
therefore $\det(\L(\G))= 0$.  Order the $n$ eigenvalues of $\L(\G)$
so that $\lambda_1=0$, then define
\begin{equation}\label{eq:defofM}
  \M(\G) = \frac{(-1)^{N-1}}N  \prod_{i=2}^N \lambda_i.
\end{equation}
Thus $\M(\G)$ is proportional to the linear term in the
characteristic polynomial of the Laplacian.  Also, $\M(\G) \neq 0$
iff $0$ is a simple eigenvalue of $\L(\G)$.
\end{define}

With this notation the Kirchhoff matrix tree theorem can be stated as
follows:

\begin{lem}[Weighted Matrix Tree Theorem]\label{lem:mtt}
  Let $\G$ be a connected, weighted graph, and $\ST(\G)$ the
  set of all spanning trees of $\G$.  Then
  \begin{equation}\label{eq:MTT}
    \M(\G) = \sum_{T\in\ST(\G)} \pi(T).
  \end{equation}
\end{lem}

\begin{remark}
  This is Theorem VI.29 in the text of Tutte~\cite{Tutte}: a proof is
  provided there.  Notice that if all of the edge weights are
  non-negative, then the sum in~\eqref{eq:MTT} is a sum of positive
  terms. This is an alternate proof that the kernel of a graph
  Laplacian with positive weights is negative semidefinite, and has a
  simple kernel, when the graph is connected.  However, once we allow
  negative weights, the sum on the right-hand side can have
  cancellations and will not be sign-definite.
\end{remark}

We will denote the vector $w = (w_+,w_-) \in (\R)^{B}\times (\R)^{R}$
by
\begin{equation*}
  (s_1,s_2,\dots, s_B, t_1, t_2, \dots, t_R).
\end{equation*}
The $s_k$ are the weights of the black edges, and the weights on the
red edges are $-t_k$.

Every spanning tree will have $N-1$ edges.  Choose a spanning tree $T$
and for some $k$, it will have $k$ red edges and $N-k-1$ black edges.
The contribution to the crossing polynomial is a term of the form
\begin{equation}\label{eq:stpi}
  \pi(T) =  (-1)^k s_{i_1}s_{i_2}\dots s_{i_{N-k-1}}t_1t_2\dots t_k.
\end{equation}
If we denote $\mathcal{ST}_k$ as the set of spanning trees with
exactly $k$ red edges, then
\begin{equation}\label{eq:stpoly}
  \M(\G) = \sum_{T\in \ST(\G)}\pi(T) = \sum_{k=0}^{N-1}\sum_{T\in \ST_k(\G)}(-1)^k s_{i_1}s_{i_2}\dots s_{i_{N-k-1}}t_1t_2\dots t_k.
\end{equation}
We see from this that $\M(\G)$ is a homogeneous polynomial of degree
$N-1$ with alternating signs.  

For the remainder of this paper, we will think of the $s_k$'s as {\em
  parameters} and the $t_k$'s as {\em variables}.  Then $\M(\G)$ is a
polynomial in the variables $(t_1,\dots, t_R)$, where the coefficients
are set parametrically by combinations of the $s_k$'s.  This
corresponds to fixing the black edges with specific weights, and
varying the red edges.

\begin{prop}
  The polynomial $\M(\G)$ has terms of degree $k$ iff
\begin{equation}\label{eq:degreebounds}
  c(G_+)-1 \le k \le  N-c(G_-).
\end{equation}
\end{prop}

\begin{proof}
  The proof is the same as the proof of Lemma 2.14
  from~\cite{Bronski.DeVille.14}, which we repeat here.  Every
  spanning tree must have at least $c(G_+)-1$ red edges to connect the
  components of $G_+$ together.  Conversely, it must have at least
  $c(G_-)-1$ black edges for the same reason.  Since every tree has
  $N-1$ edges, it has at least $c(G_+)-1$ and at most $N-c(G_-)$ red
  edges.  This gives the upper and lower bounds.

  To see that the polynomial also contains all of the intermediate
  terms, we can argue in two ways.  One argument is to use
  contraction-deletion repeatedly (see Remark~2.17
  of~\cite{Bronski.DeVille.14}) to see that the coefficients satisfy
  log-convexity; this implies that all of the intermediate terms
  appear in the polynomial as well.  A more concrete argument: we can
  note that for any number $k$ with $c(G_+)-1 \le k \le N-c(G_-)$,
  there is a subforest of $\G_-$ with exactly $k$ edges.  Since $\G_+$
  is connected, this means that we can extend this subforest to a
  spanning tree of $\G$ using only edges from $\G_+$, and thus there
  is a spanning tree of $\G$ with exactly $k$ edges, giving the term
  of degree $k$ in the polynomial.
\end{proof}

\begin{prop}\label{prop:nondec}
  All eigenvalues are nondecreasing as a function of any $t_k$.
\end{prop}

\begin{proof}
  The most straightforward proof is to use the results
  of~\cite{Bronski.DeVille.14}.  Consider the ray in $\R^R$ given by
  \begin{equation*}
    (t_1,t_2,\dots, t_R) = t(\alpha_1,\dots, \alpha_R),
  \end{equation*}
  with $\alpha_i>0$. This reduces to the framework
  of~\cite{Bronski.DeVille.14}, where the crossing polynomial was a
  function of only one independent variable, $t$.  Theorem 2.10
  of~\cite{Bronski.DeVille.14} implies that along this ray, $n_+(\G)$
  is nondecreasing, and Lemma 2.18 of~\cite{Bronski.DeVille.14}
  implies that it strictly increases by one every time we cross the
  set $\M(\G)=0$.  Inside this set, there is always a double
  eigenvalue at zero.

  A self-contained argument is as follows: First note that for any
  $x$,
  \begin{equation*}
    x^t \L(\G)x = -\sum_{(i,j)\in E} \gamma_{ij}(x_i-x_j)^2.
  \end{equation*}
  Clearly, if we increase any $t_k$, then this will decrease one of
  the $\gamma_{ij}$, thus increasing this quadratic form.  From the
  Courant minimax theorem, this implies that the eigenvalues of
  $\L(\G)$ are nondescreasing as a function of $t_k$, since all
  eigenvalues can be written as the maximum of this quadratic form
  restricted to some subspace.
\end{proof}

\subsection{Coefficients of polynomial}

In Definition~\ref{def:M}, we noted that $\L(\G)$ has a multiple zero
eigenvalue iff $\M(\G)=0$, and from Proposition~\ref{prop:nondec} this
means that eigenvalue crossings occur on the zero variety $Z(\G) =
\{\M(\G) = 0\}$ in $\R^R$, i.e. the connected components of $\Z(\G)^c$
in $\R^R$ are exactly the sets where $n_+(\G)$ is constant.

\newcommand{\pit}{\widetilde\pi}
\begin{define}
  For any spanning tree $T$ of $\Gamma$, we define $\pit(T)$ as the
  product of the weights of all of the black edges in $T$, i.e.
\begin{equation*}
  \pit(T) = \prod_{\substack{i<j,\\ (i,j)\in E(T),\\ \gamma_{ij}>0}}\gamma_{ij}.
\end{equation*}
\end{define}

\newcommand{\sij}[2]{S_{{#1}}^{{#2}}}
\newcommand{\ssij}[2]{\mathcal S_{{#1}}^{{#2}}}
\newcommand{\si}[1]{S_{{#1}}}
\newcommand{\ssi}[1]{\mathcal S_{{#1}}}

\begin{define}\label{def:SIJ}
  Let $\Gamma$ be a weighted graph with $R$ negative edges, and let
  $I,J$ be disjoint subsets of $[R]$.  We define $\sij I J$ as the set
  of all spanning trees of $\Gamma$ which contain all of the red edges
  in the index set $I$, and none of the red edges in index set $J$,
  i.e.
\begin{equation*}
    \sij IJ(\Gamma) := \left\{T \ |\ T\mbox{ is spanning tree of $\Gamma$}, e_i\in T \ \forall i\in I, e_j\not\in T\ \forall j\in J\right\},
\end{equation*}
and we define 
\begin{equation*}
  \ssij IJ(\Gamma) := \prod_{T \in \sij IJ}\pit(T).
\end{equation*}
Finally, as shorthand, we will define
\begin{equation*}
  \si I = \sij I{I^c},\quad \ssi I = \ssij I {I^c},
\end{equation*}
where the complement $I^c = [R]\setminus I$.
\end{define}

\begin{define}[Deletion and contraction]\label{def:dc}
  Let $\G = (V,E)$ be a weighted graph, and $e\in E(\G)$ an edge.  

\begin{itemize}
\item We denote by $\G\setminus e$ the graph obtained by removing edge
  $e$, and call this the {\bf deletion} of edge $e$ from $\G$.

\item If $e = (v_1,v_2)$ is an edge with $v_1\neq v_2$, we define
  $\G.e$ as follows: identify the two vertices $v_1$ and $v_2$ as
  a single vertex $v^*$; for any vertex $w$ connected to $v_1$ or
  $v_2$, we define the new edge weight $\gamma_{v^*,w} =
  \g_{v_1,w}+\g_{v_2,w}$.  We call this the {\bf contraction} of edge $e$ in $\G$.
\end{itemize}
\end{define}

\begin{define}
  Let $\G$ be a graph with $R$ red edges.  If $I,J$ are two disjoint
  subsets of $\{1,\dots, R\}$, then the graph $\G_I^J$ is the graph
  obtained by first deleting all of the edges in $I$, and then
  contracting all of the edges in $J$. As above, we write $\G_I$ as
  shorthand for $\G_I^{I^c}$.
\end{define}

\begin{prop}
Using this notation,
\begin{equation}\label{eq:ssi}
  \ssi I (\Gamma) = \M(\Gamma_I).
\end{equation}

\end{prop}

The proof is the standard contraction-deletion theorem, q.v.~\cite[\S
13.2]{Godsil.Royle}.

\begin{remark}
  Let $\G$ be a graph with $R$ red edges.  For any $I\subseteq [R]$,
  the coefficient of the term $t^I$, i.e. $\prod_{k=1}^{\av{I}}
  t_{i_k}$, is $\ssi I(\Gamma)$.  In particular, this allows us to use
  the compact notation
  \begin{equation}\label{eq:ssipoly}
    \M(\G) = \sum_{I \in 2^R} \ssi I(\Gamma) t^I.
  \end{equation}

  If we constrain all positive weights to be one, the quantity $\ssi
  I(\Gamma)$ has a the combinatorial interpretation as the number of
  spanning trees that contain all red edges in $I$ and do not contain
  any red edges not in $I$.
\end{remark}

\begin{remark} We do not use~\eqref{eq:ssipoly} explicitly in the
  results below, but is quite powerful computationally: given $\Gamma,
  I$, to compute $\ssi I(\Gamma)$, we can compute $\Gamma_I$, write
  down its graph Laplacian $\mathcal{L}(\Gamma_I)$, then compute the
  linear term of its characteristic polynomial.  This process can be
  defined solely in terms of matrices, is easily automatable using a
  computer algebra system, and is relatively inexpensive.
\end{remark}

Finally, we state the main theorem of the paper:

\begin{thm}\label{thm:generic}
  For generic positive weights, all red rays give single eigenvalue
  crossings.  For all positive weights, generic red rays give single
  eigenvalue crossings.
\end{thm}

We prove these theorems in Sections~\ref{sec:R2} and~\ref{sec:R>2}
below.

\section{Multiple eigenvalue crossings, if $R=2$}\label{sec:R2}

In this section, we concentrate on the case where there are exactly
two red edges (see Section~\ref{sec:R>2} for the case of more than two
red edges).  In this section, we write the vector $(x_1,x_2)$ as
$(x,y)$.

Let us assume that $\Gamma_+$ is connected.  Then the main theorem
of~\cite{Bronski.DeVille.14}, Theorem 2.10, implies that as the
weights on the red edges go to zero, $n_+(\G) = 0$, and if they are
sufficiently large, $n_+(\G) = 2$.  Only two eigenvalues cross through
zero, so either they cross as a degenerate pair, or they do not.  In
this section we compute the conditions that determine this.

If $\Gamma_+$ is not connected, then the crossing polynomial will not
have a constant term, but after factoring out the lowest-order terms,
it will be in the form~\eqref{eq:mg2}, so that analogous statements
about genericity also hold.

\subsection{Discriminant}\label{sec:disc2}

Here we connect the discriminant of the 2-variable crossing polynomial
to the combinatorics of the graph.  Recall from before that we have
\begin{equation}\label{eq:mg2}
  \M(\G) = A_{11} xy  - A_{10} x - A_{01} y + A_{00},
\end{equation}
where $A_I = \ssi I(\Gamma)$.  Since $\G_+$ is connected, $A_{00}>0$.

\begin{define}
  The {\bf discriminant} of the polynomial~\eqref{eq:mg2} is the quantity
  \begin{equation*}
    \Delta = A_{11}A_{00} - A_{01}A_{10}.
  \end{equation*}
\end{define}

\begin{lem}\label{lem:hyperbola}
  If $\Delta \neq 0$, the zero-set $Z(\G)$ defines a hyperbola, with
  the minimum Euclidean distance $d$ between the branches given by
  \begin{equation}\label{eq:Euc2}
    d = \frac{\sqrt{2\Delta}}{A_{11}}.
  \end{equation}

  In the case $\Delta=0$ the set $Z(\Gamma)$ is a reducible variety
  and degenerates to the union of the lines $x=A_{01}/A_{11}$ and
  $y=A_{10}/A_{11}$.
\end{lem} 

\begin{remark}
This clearly illustrates the level-repulsion phenomenon: while we can make one 
eigenvalue vanish by varying the strength of one of the edges (say 
$x$) we typically {\em cannot} make two eigenvalues vanish by varying the 
strength of two edges. The typical situation, when $\Delta\neq 0$, is that 
there are two disconnected branches of the zero-set, each corresponding to 
{\em a single zero eigenvalue}. It is only in the degenerate case where
 $\Delta=0$ that we are able to create a double zero eigenvalue. 
\end{remark}

\begin{figure}[ht]
\begin{centering}
  \includegraphics[width=0.45\textwidth]{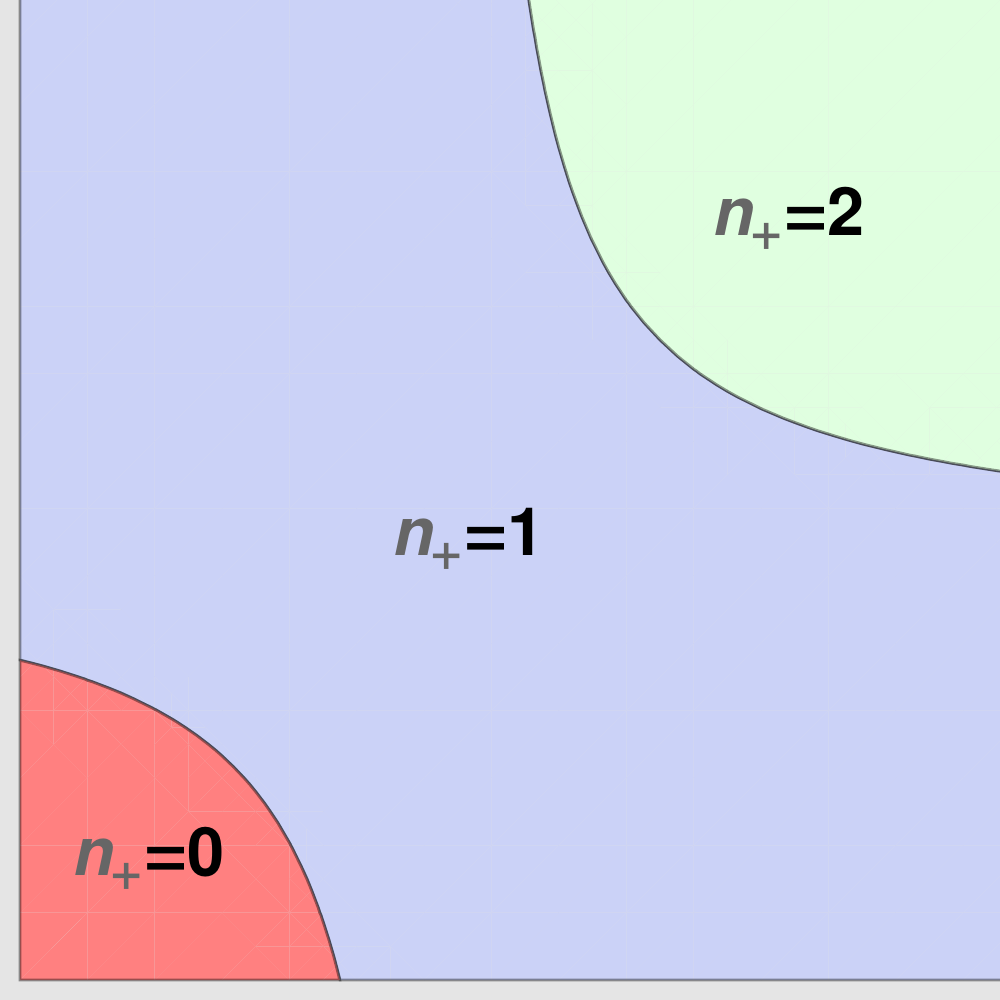}
  \includegraphics[width=0.45\textwidth]{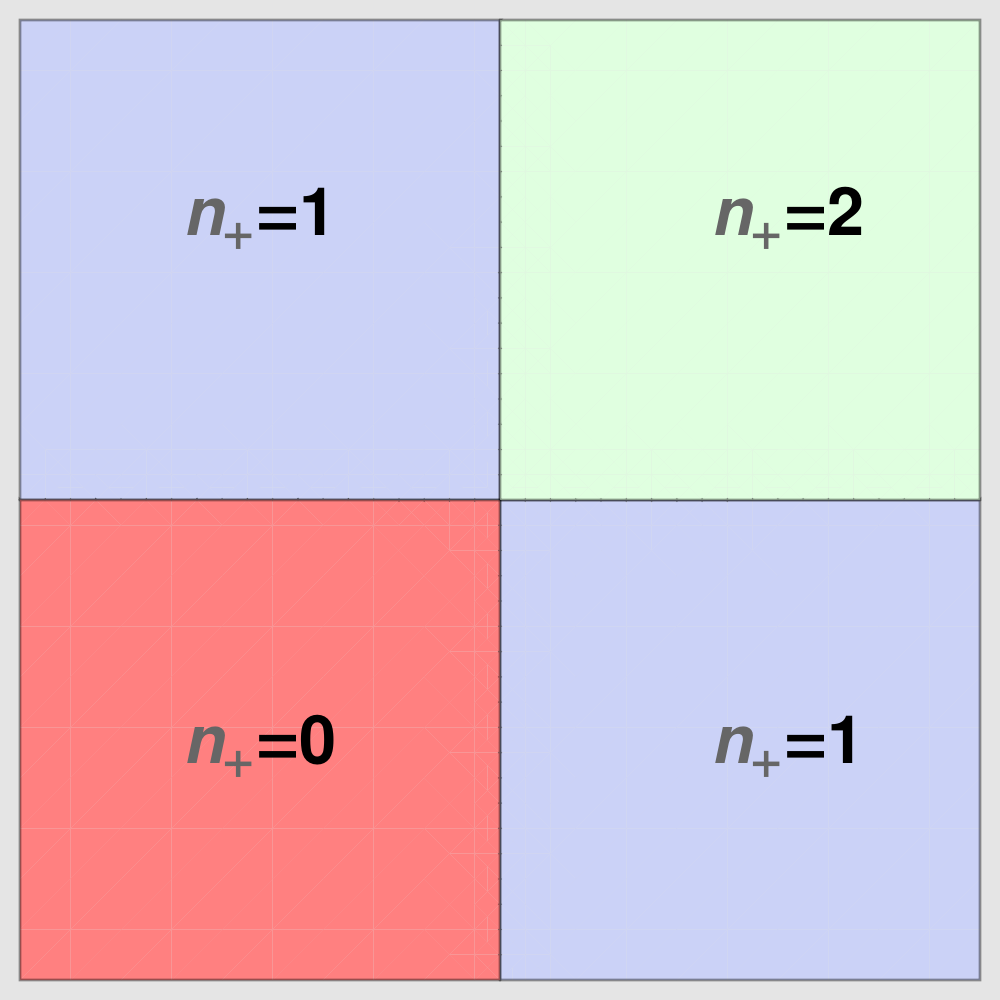}
  \caption{A schematic diagram of the nondegenerate and degenerate
    cases when there are two red edges.}
  \label{fig:schematic}
\end{centering}
\end{figure}

\subsection{Proof of Theorems~\ref{thm:generic} for $R=2$}
It is clear from these formulae that the presence of a potential
double root is nongeneric, since the weights of the matrix need to
satisfy a polynomial equation.  Thus for $R=2$, we have proved
Theorem~\ref{thm:generic}: as long as $\Delta\neq 0$, then all rays
give single eigenvalue crossings, and this is clearly generic.
Moreover, even if $\Delta = 0$, the only way to actually obtain the
double eigenvalue crossing is along the ray that passes through the
point $(A_{01}/A_{11}, A_{10}/A_{11})$.

From this, it is clear that a double eigenvalue is doubly nongeneric:
we need the discriminant to be zero, and even then, we need to choose
the correct ratio of weights on the red edges to hit the sweet spot.
See Figure~\ref{fig:schematic}.

\subsection{Counting forests}

We have proved the main theorems by showing that the discriminant is
the quantity that controls how closely the eigenvalues can cross We
now give a combinatorial interpretation to this discriminant.

\begin{define}
  For any graph, a {\bf spanning 2-forest} is a spanning subgraph with
  two connected components --- in other words, a spanning tree with
  one edge removed.  More generally, a {\bf spanning $k$-forest} is a
  spanning subgraph with $k$ connected components.
\end{define}

\newcommand{\ft}[2]{F^{(2)}_{{#1},{#2}}}

\newcommand{\fk}[3]{\mc{F}^{({#1})}_{{#2},{#3}}}

\begin{define}
  Let $G=(V,E)$ be a signed graph with $\av{V} = n$.  Let $U,W$ be
  subsets of $[n]$ with $\av{U} = \av{W} = k$.  We define $\fk k U W$
  as the set of all spanning $k$-forests such that every tree in the
  forest has exactly one vertex in $U$ and one vertex in $W$.  

  We can associate a sign to each forest in $F\in\fk k U W$, denoted
  $\e(F)$, as follows.  Choose and fix some enumeration for $U,W$, and
  define the map $g\colon W\to U$ by saying that $u = g(w)$ is the
  unique element of $U$ that is in the same component of $F$ as $w$.
  Then $g(W)$ is a permutation of $U$ and thus has a sign, and this is
  $\e(F)$.

\end{define}

\begin{thm}
\label{thm:2-forest}
Let $\Gamma$ be a weighted graph with two red edges and where the red
edges have weight $x$ and $y$. Let $i,j$ be the vertices of the first
red edge, and $k,l$ be the vertices of the second.  We can form the
polynomial~\eqref{eq:mg2} as above.  If we choose $U=\{i,j\}$ and
$W=\{k,l\}$, then
\begin{equation}\label{eq:Delta1}
  \Delta = - \left(\sum_{F \in \ft U W }\epsilon(F)\pi(F) \right)^2 
\end{equation}
where $\pi(F)$ is defined as in~\eqref{eq:defofpi}.  If all of the
black edges have integral weight, then this and~\eqref{eq:Euc2}
implies that the Euclidean distance between the two branches of
$Z(\Gamma)$ is $\sqrt2$ times a rational number.
\end{thm}

To prove this theorem, we first need some auxiliary results.

\begin{thm}[Chaiken~\cite{Chaiken.82}]
  Let $G=(V,E)$ be a signed graph with $\av{V} = n$.  Let $U,W$ be
  subsets of $[n]$ with $\av{U} = \av{W} = k$.  Define $|\mc L_{U,W}|$
  to be the minor determinant formed by removing the rows $U$ and the
  columns $W$ from the graph Laplacian $\mc{L}(G)$. Then
  \begin{equation}\label{eq:supersum}
    |\mc L_{U,W}| = (-1)^{\l(\sum_Uu + \sum_Ww\r)} \sum_{F \in \fk k U W} \sigma(F).
  \end{equation}
\end{thm}

\begin{proof}
  See Chaiken~\cite[\S 3]{Chaiken.82} for a proof in the general
  (directed graph) case.  He uses this result to prove the directed
  MTT, the directed generalization of~\eqref{eq:MTT} above.
\end{proof}

\begin{remark}
  The case we are particularly interested in here is $\av{U} = \av{W}
  = 2$.  In this case, forests which assign the least element of $U$
  and the least element of $W$ to the same component are counted $+1$,
  and forests which assign the least element of $U$ and the least
  element of $W$ to different components are counted $-1$.
\end{remark}

\begin{remark}
It is not hard to see that this theorem implies that there are a number 
of linear relationships among the minor determinants of a graph Laplacian. 
One example is
\begin{equation}\label{eq:rel1}
 \av{\L_{12,13}} + \av{\L_{12,14}} = \av{\L_{12,34}}.  
\end{equation}
To see this note that the first term, $\av{\L_{12,13}}$, is equal to
minus the number of 2-forests with vertex $1$ in one component and
vertices $2$ and $3$ in the other component. This collection of
2-forests can be split into two types: those in which vertex $4$ is in
the component with vertex $1$ and those in which vertex $4$ is in the
component with vertices $2$ and $3$.  Similarly the second term,
$\av{\L_{12,14}}$, is equal to the number of 2-forests with vertex $1$
in one component and vertices $2$ and $4$ in the other component. This
collection of 2-forests can similarly be split into two types: those
in which vertex $3$ is in the component with vertex $1$ and those in
which vertex $4$ is in the component with vertices $2$ and $3$.  The
terms in which vertices $2$, $3$ and $4$ share a component cancel, and
we are left with the number of 2-forests in which $1$ and $3$ share a
component and $2$ and $4$ share a component, minus the number in which
$1$ and $4$ share a component and $2$ and $3$ share a component. This
is exactly $\av{\L_{12,34}}$.  A similar identity, which we will use
in the proof, is
\begin{equation}
  \av{\L_{13,13}} + \av{\L_{14,14}} + \av{\L_{13,23}} + \av{\L_{14,23}} = \av{\L_{12,34}}.
\end{equation}

\end{remark}
The next result we use is an identity variously attributed to
Dodgson~\cite{Dodgson.1866}, Jacobi~\cite{Bressoud.99}, and
Desnanot~\cite{Muir.1906} on determinants of minor determinants. There
is a nice bijective proof due to Zeilberger~\cite{Zeilberger.98}.

\begin{thm}[Dodgson]
  Let $M$ be a matrix, and $M_{U,W}$ denote the submatrix formed by
  deleting the rows in $U$ and the columns in $W$. Then
  \begin{equation*}
    \av{M} \av{M_{ij,kl}} - \av{M_{i,k}}\av{M_{j,l}} = -\av{M_{i,l}M_{j,k}}.
  \end{equation*}
  
\end{thm}

This identity underlies Dodgson's Method of Condensation, an
algorithm for efficient hand computations of determinants. We are now
in a position to prove the main theorem:

{\bf Proof of Theorem~\ref{thm:2-forest}.}  For this proof, we will
use the symbol $\L_{U,W}$ to be the determinant of the Laplacian
matrix with rows $U$ and columns $W$ removed.

Let us first assume that the red edges $x$ and $y$ do not share a
vertex.  Renumber the vertices so that edge $x$ connects vertices $1$
and $2$ and edge $y$ connects vertices $3$ and $4$.  First note that
any 1-minor determinant is (up to sign) $P(x,y)$.  In particular we
have that
\begin{equation}\label{eq:L13}
  P(x,y) = \mc L_{1,3} = A_{00} + A_{10}x+A_{01}y+A_{11}xy.
\end{equation}
Note that we also have
\begin{equation}
  \mc L_{12,13} + \mc L_{12,23} = A_{10} + A_{11} y 
\end{equation}

and that 
\[
\L_{13,34} + \L_{14,34} = A_{01} + A_{11} x 
\]
and finally that 
\[
\L_{123,134} + \L_{124,134} + \L_{123,234} + \L_{124,234} = A_{11}. 
\]
Some algebra gives that
\begin{equation*}
A_{11} ( A_{00} + A_{10} x + A_{01} y + A_{11} x y ) -
(A_{10} + A_{11} y)(A_{01} + A_{11} x) = A_{11}A_{00} - A_{10}A_{01} =
\Delta,  
\end{equation*}
or
\begin{equation*}
  \Delta = (\L_{123,134} + \L_{124,134} + \L_{123,234} + \L_{124,234})\L_{1,3}-(\mc L_{12,13} + \mc L_{12,23})(\L_{13,34} + \L_{14,34})
\end{equation*}
and applying the Dodgson identity (here $\L_{1,3}$ is the matrix $M$)
gives
\[
\Delta = -\L_{12,34}\left(\L_{13,13} + \L_{14,13}+\L_{13,23}+\L_{14,23}\right)=-(\L_{12,34})^2 
\] 

In the case where $x$ and $y$ share a vertex we number the shared vertex 
$1$ and the other two $2$ and $3$. Then a similar calculation to the one 
above gives 
\[
\Delta = - (\L_{12,13})^2 
\] 
\qed
 
\begin{corollary}\label{cor:auto}
  Assume $\Gamma$ has two disjoint red edges, $x$ and $y$, and denote
  $x=(i,j)$ and $y=(k,l)$.  If there is an automorphism of $\Gamma$
  that exchanges $i,j$ and fixes $k,l$ (or vice versa), then $\Delta =
  0$.
\end{corollary}

\begin{proof}
  This automorphism fixes $\pi(F)$ but negates $\epsilon(F)$, thus giving $\Delta = -\Delta$.
\end{proof}

\subsection{Cycles and homology}

The expressions derived above give explicit formulae in terms of sums
over spanning trees and spanning 2-forests; we now show that there is
a dual formula in terms of minor determinants of the cycle
intersection form for the homology of the graph $\Gamma$.

\begin{thm}\label{thm:Cauchy-Binet}
  Suppose that the original graph $\Gamma$ has co-rank $c=E-n+1$, and
  that removing the edges $x$ and $y$ does not disconnect $\Gamma$.
  We construct a cycle basis for the graph in the following way: let
  $v_1, v_2, \ldots v_{c-2}$ denote a cycle basis for the graph with
  edges $x$ and $y$ removed, $v_{c-1}$ denote any cycle through $x$
  {but not through} $y$ and $v_c$ denote any cycle through y { but not
    through} $x$.

  In other words the basis should be chosen so that there is a unique
  cycle through each of the edges $x$ and $y$, but the basis is
  otherwise arbitrary.  Next we define the matrix $F$ to be the
  $E\times c$ matrix with rows given by $\{v_i\}_{i=1}^c$. Let $M_i^j$
  denote the minor determinant formed by removing the $i^{th}$ row and
  the $j^{th}$ column of $FF^t$. Then $\Delta = |M_c^{c-1}|^2$.
\end{thm} 

\begin{proof}
  The proof here is analogous to a similar result due to
  Sjogren\cite[Theorem 1]{Sjogren.91} that shows that the determinant 
of $FF^t$ is the number of spanning trees of the graph. The proof here is
  similar in flavor, except that, as we are dealing with a minor determinant, 
 the basis must be chosen in the particular way described above.

 Having chosen the basis as above, the main idea is the Cauchy--Binet
 formula\cite[\S0.8.7]{Horn.Johnson.13}.  We will have to consider
 various submatrices of $F$. Superscripts will denote deleted rows of
 the matrix (at most one row will be deleted), and subscripts will
 denote the retained columns. For instance if $S$ is a subset of the
 edges then $F^{c-1}_S$ will denote the submatrix obtained from $F$ by
 deleting the $(c-1)^{st}$ row (corresponding to the cycle through
 edge $x$) and deleting all columns not in $S$.

 Then the Cauchy--Binet formula~\cite[\S0.8.7]{Horn.Johnson.13} gives
 that $|M_c^{c-1}| = \sum_{S} |F^{c-1}_{S}| |F^c_{S}|,$ where the sum
 is over all subsets $S$ of the edge set $E$ of size $c-1$.

Using Lemma~\ref{lem:facts}, the remainder of the proof is
straightforward: $|F^{c-1}_{S}| |F^{c}_{S}|$ unless $\Gamma/(S\cup y)$
and $\Gamma/(S\cup x)$ are both trees. This is clearly equivalent to
$\Gamma/(S\cup y \cup x)$ being a spanning 2-forest where $x$ and $y$
each have one vertex in each component, so $|M_c^{c-1}|$ reduces to a
signed sum over spanning 2-forests with this property, with the sign
given as in the theorem.
\end{proof}

\begin{lem}\label{lem:facts}
  Using the notational conventions established in the statement and
  proof of Theorem~\ref{thm:Cauchy-Binet}, we have
\begin{itemize}
\item $ |F^{c-1}_{S}|=0,\pm1.$ The determinant is zero unless $\Gamma/(S\cup y)$ is a tree, in which case $ |F^{c-1}_{S}|=0,\pm1,$  and similarly for $ |F^{c}_{S}|=0$ .
\item  In the case where   $E/(S\cup y)$ and  $E/(S\cup x)$ are both trees then 
the sign of $ |F^{c-1}_{S}| |F^{c}_{S}|$ is determined as follows: $E/S$ is a graph with 
a unique (up to sign) cycle that contains edges ``x'' and ``y''. If this 
cycle traverses edges ``x'' and ``y'' in the same sense then $ |F^{c-1}_S| |F^{c}_{S}|=1$, otherwise 
$ |F^{c-1}_{S}| |F^c_{S}|=-1.$ Note that this sign is 
independent of the orientation assigned to edges ``x'' and ``y'' and the 
direction of the cycle through them.   
\end{itemize} 
\end{lem}

\begin{proof}
  Let $c=E-N+1$ be the co-rank of the graph.  We first note the
  following: if $\tilde S$ is a set of $c$ edges of $\Gamma$ then
  $|F_{\tilde S}|=\pm1$ if $\Gamma/\tilde S$ forms a tree and
  $|F_{\tilde S}|=0$ otherwise. This is proven by
  Sjogren~\cite{Sjogren.91}, but for the convenience of the reader we
  give a quick proof here. First, to see that $|F_{\tilde S}|=0$ if
  $\Gamma/\tilde S$ is not a tree we give an explicit element of the
  kernel.  Since $\Gamma/\tilde S$ is not a tree it must contain a
  cycle. Express this cycle in terms of the basis as $\sum \alpha_i
  v_i$.  Thus $\vec \alpha F_{\tilde S}$ gives the edges of this cycle
  lying in $\tilde S$. Since this cycle lies entirely in the
  complement, this is zero.  

  To see that $|F_{\tilde S}|=\pm1$ when $\Gamma/\tilde S$ is a tree
  $T$ we first consider the case where the basis is chosen in a
  special way: namely for each edge $i$ in $\tilde S$ we consider the
  unique cycle in $T \cup {i}$. This gives a cycle basis. In this
  particular basis the matrix $F_{\tilde S}$ is a permutation matrix,
  since each edge in $\tilde S$ lies in a unique cycle.  Now {\em any}
  integral basis is related to this special basis by a matrix $U$ with
  $|U|=\pm 1$, and the corresponding matrices are related by $\tilde
  F_{\tilde S}= U F_{\tilde S}.$ Therefore for any basis $|F_{\tilde
    S}|=\pm1.$

  We are interested in the $(c-1)\times(c-1)$ matrices $F^{c}_{S}$ and
  $F^{c-1}_{S}$, where $S$ is now a subset of $c-1$ edges, but these
  matrices can be realized as minors of a matrix $F_{\tilde S}$ of the
  above form. It is clear that if $S$ contains the edge $x$ then
  $|F^{c-1}_S|=0$, since we have a row of all zeros, and similarly for
  $|F^{c}_S|$. Since we are only interested in the case where the
  product is non-zero we can therefore just consider the the case
  where $S$ does not contain either $x$ or $y$.  By the above the
  matrix $|F_{S\cup y}|=\pm1$ if $\Gamma/(S\cup y)$ is a tree, and
  zero otherwise.  A minor expansion in the row corresponding to $y$
  has a single entry of $1$ (in the $(c,c)$ position), and the rest of
  the entries are zero. The minor corresponding to the sole nonzero
  entry is $F^{c}_S$, giving $|F^{c}_S|=|F_{S\cup y}|=\pm1$. Similarly
  we have that $|F^{c-1}_S|=-|F_{S\cup x}|=\pm1$, the minus sign
  arising since the nonzero entry is in the $(c-1,c)$ position.

  To evaluate the relative sign between $|F^{c}_{S}|$ and
  $|F^{c}_{S}|$ we give a series of elementary row operations to
  reduce one to the other.  We let $S$ be an edge subset as above and
  $G$ be the $c\times (c+1)$ matrix consisting of the columns of $F$
  from $S\cup y\cup x$.  Since $G$ has a $c\times c$ submatrix of full
  rank it has a one dimensional kernel, which can be expressed
  naturally in terms of the cut-space: Removing the edges $S\cup x
  \cup y$ from the graph gives a graph with two components, denoted
  $A$ and $B$. For each edge in $S\cup x \cup y$ we associate the
  following vector $w$: if the edge points from $A$ to $B$ we assign
  the corresponding entry of $w$ to be $+1$, if the edge points from
  $B$ to $A$ we assign $-1$, otherwise we assign the entry $0$.  This
  vector must be in the kernel of $G$, since each entry of $G w$ gives
  the number of times that the cycle leaves $A$ minus the number of
  times that it leaves $B$. This allows us to express the $c+1$ column
  of $G$ in terms of the first $c$ columns of $G$, and thus gives a
  sequence of elementary row operations to reduce $|F^{c-1}_{S}|$ to
  $|F^{c}_{S}|$. The only one which influences the sign of the
  determinant $G$ is the orientation of the edges $x$ and $y$ relative
  to the components $A$ and $B$: if the edges $x$ and $y$ both point
  from $A$ to $B$, or vice verse, then $|F^{c-1}_{S}|=-|F^c_{S}|,$ if
  one edge points from $A$ to $B$ and the other $B$ to $A$
  then$|F^{c-1}_{S}|=|F^c_{S}|.$ Equivalently $\Gamma/S$ has a cycle
  containing both edges $x$ and $y$.  If this cycle traverses both
  edges in the same sense the determinants have the same sign,
  otherwise they have opposite signs.

\end{proof}

\begin{example}
\label{exa:complete}
One graph for which it is straightforward to compute everything is the
complete graph, $K_n$. We consider the case where the graph has two
red edges ($x$ and $y$) and the rest of the edges have weight $1$.  In
this case it is clear that there are two topologically distinct
situations, the case where the two edges share a vertex and the case
where they do not. In the case where the two red edges share a vertex
it is straightforward though somewhat tedious to check that the
polynomial is given by
\[
P_{K_n}(x,y) = 3 n^{n-4} x y + (2n-3)n^{n-4} x + (2n-3)n^{n-4} y + (n-1)(n-3)n^{n-4}
\]
and the discriminant is then $\Delta_{K_n} = n^{2n-6}$.

In the second case, where the edges do not share a vertex, the polynomial 
is given by 
\[
P_{K_n}(x,y) = 4 n^{n-4} x y + (2n-4)n^{n-4} x + (2n-4)n^{n-4} y + (n-2)^2n^{n-4}
\]
and the corresponding discriminant vanishes: $\Delta_{K_n} = 0$.  (Of
course, this vanishing is implied by Corollary~\ref{cor:auto}; there
are many automorphisms of the complete graph which flip one red edge
and leave the other fixed.)
\end{example}

\section{Multiple eigenvalue crossings --- $R>2$}\label{sec:R>2}

\newcommand{\w}{\smiley}

Let us now consider the case where there are $R$ red edges, $R>2$.  As
above, the crossing polynomial~\ref{eq:stpoly} is of the form
\begin{equation*}
  \M(\G) = \sum_{I\in 2^R} \ssi I (\Gamma) x^I.
\end{equation*}
We will find it convenient to index subsets of $2^R$ by the
corresponding binary sequence, and we will write $A_I = \ssi I
(\Gamma)$ for brevity.

The variety $Z(\G) := \{x\in\R^R\colon M(\G)(x) = 0\}$ can be
reducible in many different combinations, but here we concern
ourselves only with the maximal notion of reducibility, i.e. we
consider the case when $Z(\G)$ can be written as the union of
hyperplanes and the crossing polynomial
can be written as a product of linear factors.

It is clear that this is a very non-generic case: a general 
polynomial $M(\Gamma)$ is determined by $2^R$ different coefficients, 
whereas an $M(\Gamma)$ that factors is determined by $R+1$ different
coefficients. This suggests that a necessary and sufficient condition 
for $M(\Gamma)$ to factor should be the vanishing of $2^R-R-1$ 
functions of the coefficients. The purpose of this section is to prove 
this result, and give a nice description of the $2^R-R-1$ discriminants 
whose vanishing is equivalent to factorization of $M(\Gamma).$

\begin{define}
  A {\bf wildcard} sequence is any binary sequence on $2^R$ with any
  two of the digits replaced by $\w$'s.  Let $w$ be such a wildcard,
  and it corresponds with a set of four binary sequences in $2^R$ in
  the canonical way: simply replace the two $\w$'s by the four
  sequences $\{00,01,10,11\}$, call this set $S_w \subseteq 2^R$.  For
  any wildcard $w$, we define
  \begin{equation*}
    P_w(x) = \sum_{b\in S_w} A_b x^b.
  \end{equation*}
  We also define $\Delta_w$, the discriminant of $P_w$, in the obvious
  manner: it is the product of the coefficients of the highest and
  lowest order terms, minus the product of the coefficients of the
  terms of middle order.  
\end{define}

\begin{example}
  If $R=5$ and $w = 00\w1\w$, then 
\begin{equation*}
  P_w(x) = A_{00010}x_4 + A_{00011}x_4x_5 + A_{00110}x_3x_4 + A_{00111}x_3x_4x_5,
\end{equation*}
and
\begin{equation*}
  \Delta_w = A_{00010}A_{00111} - A_{00011}A_{00110}.
\end{equation*}
\end{example}

The main result of this section is that the crossing polynomial
factors into a product of linear factors if and only if ``enough'' of
the reduced discriminants are zero, and this corresponds to the
combinatorial computations of the previous section in a
straightforward manner.  We first show one direction of this theorem:

\begin{prop}
  If $\M(\G)$ factors into $R$ linear factors, then $\Delta_w = 0$ for
  any wildcard $w$.
\end{prop}

\begin{proof}
  We assume that the crossing polynomial factors, so we factor it and
  then pull out all of the constants in the following manner:
  \begin{equation*}
    \M(\G)(x) = \alpha \prod_{i=1}^R (C_i x_i + 1).
  \end{equation*}
  We think of $C = (C_1,\dots C_R)$ as a vector, and given any vector
  $p=(p_1,\dots, p_R)$, we write
  \begin{equation*}
    C^p = C_1^{p_1}\cdot C_2^{p_2}\cdots C_R^{p_R}.
  \end{equation*}
  From this notation, we see that for any $b\in 2^R$, $A_b = \alpha
  C^b$.

  Let $w$ be a wildcard, and denote the positions of the $\w$'s by
  $i<j$.  Let $b$ be the binary sequence corresponding to $w$ where we
  replace both $\w$'s by $0$'s, and then the four binary sequences
  correspond to $w$ are $\{b, b+e_i, b+e_j, b+e_i+e_j\}$, where $e_i$
  are the standard basis vectors.  Then
  \begin{equation*}
    \Delta_w = C^{b+e_i+e_j}C^b - C^{b+e_i}C^{b+e_j}
  \end{equation*}
  and this is clearly zero.
\end{proof}

\begin{remark}
The previous proposition shows that a complete factorization of $M(\Gamma)$ 
implies that $\binom{R}{2} 2^{R-2}$ discriminants vanish. While it is not 
hard to see that the converse is true, we can prove something stronger.  
We expect that not all of these discriminants are independent, and the naive 
count suggests that we need only $2^R-R-1$ conditions in order to guarantee 
that $M(\Gamma)$ factors.  In the next part we define a subset of $2^R-R-1$ 
whose vanishing guarantees that  $M(\Gamma)$ factors.

\end{remark}
\begin{define}
  For any $R$, we define $W_R$ (a stacked deck) to be the set of all 
  wildcards of length $R$ satisfying the following properties:
  \begin{itemize}
  \item The two $\w$'s may be placed in any positions.
  \item All of the bits {\em before} the second $\w$ must be zero. 
  \item The bits {\em after} the second $\w$ are unconstrained: they
    may be either $0$ or $1$.
  \end{itemize}
\end{define}

\begin{prop}\label{prop:wildcard}
  The set $W_R$ has $2^R-R-1$ elements.
\end{prop}

\begin{proof}
  Let $w$ be a wildcard where the second $\w$ is in slot $j$.  Then the
  first $\w$ can be chosen in any of $j-1$ positions, and all of the
  entries after $j$ are free.  Thus there are $2^{R-j}(R-1)$ wildcards
  where the second $\w$ is in the $j$th slot, and therefore
  \begin{equation*}
    \av{W_R} = \sum_{j=2}^R (j-1) 2^{R-j} = 2^R-R-1.
  \end{equation*}
\end{proof}

\begin{example}
For $R=3$, we have  $|W_3|=8-3-1=4$. The elements of $W_3$ are 
\begin{equation}
  W_3 = \{\w\w 1 , \w\w0 , \w0\w, 0\w\w\}.
\end{equation}
For $R=4$, we have $|W_4|=11$. The elements of $W_4$ are 
\begin{equation}
  W_4=\{\w\w11 , \w\w10, \w\w01, \w\w00, \w0\w0, \w0\w1, \w00\w, 0\w\w0, 0\w\w1, 0\w0\w, 00\w\w\}
\end{equation}
\end{example}

\begin{lem}\label{lem:wcrecursion}
  Every sequence in $W_R$ is of two types:
  \begin{itemize}
  \item it is of the form $w0$ or $w1$ for some $w\in W_{R-1}$,
  \item it is length $R$, it ends with a $\w$, and all non-$\w$
    entries are zeros.
  \end{itemize}
\end{lem}

\begin{proof}
  It is clear that any sequence described above is in $W_R$.  Then we
  simply need to count all of the entries described above.  There are
  $2W_{R-1}$ of the first type of sequence, and $R-1$ of the second
  type.  Thus we have obtained $2W_{R-1} + R-1$ sequences of length
  $R$, and this the same recursion relation obtained by the sequence
  $|W_R|_{R=3,\dots}$.
\end{proof}

\begin{thm}\label{thm:wildcard}
  If $\Gamma$ is a weighted graph with $R$ negative edges, then a
  necessary and sufficient condition for $Z(\G)$ to decompose as a
  union of hyperplanes is that $\Delta_w=0$ for all $w \in W_R.$
\end{thm}

\begin{proof}
  The necessary direction of this statement has been proved already in
  Proposition~\ref{prop:wildcard}, so we consider sufficiency here.
  We will use a proof by induction, with base case $R=2$; this case
  was established in Lemma~\ref{lem:hyperbola}.

  Assume that the theorem is true for $R-1$, and that $\Delta_w = 0$
  for all $w\in W_R$.  Let $b_R \subseteq 2^R$ be all of those
  sequences with a one in the last slot, and let $P_R\colon 2^R \to
  2^{R-1}$ be the projection on sequences that forgets the last slot.
  Then
\begin{equation*}
  \M(\G) = \sum_{I\in 2^R} \ssi I(\G) x^I =  x_R\sum_{I\in b_R}A_I x^{P_R(I)}+\sum_{I\not\in b_R}A_I x^{P_R(I)}
\end{equation*}
We write
\begin{equation}\label{eq:defoffg}
  f_R(x) = \sum_{I\in b_R}A_I x^{P_R(I)},\quad g_R(x) = \sum_{I\not\in b_R}A_I x^{P_R(I)},
\end{equation}
giving 
\begin{equation*}
  \M(\G) = x_R f_R(x) + g_R(x).
\end{equation*}
Note that $f_R(x), g_R(x)$ are functions only of $x_1,\dots,x_{R-1}$.

We can write $f_R(x)$ as in~\eqref{eq:defoffg}, but we can also write
this as
\begin{equation*}
  f_R(x) = \sum_{I\in 2^{R-1}} A_{I1}x^I.
\end{equation*}
Note then that the condition that $f_R(x)$ fully factor is then that
$\Delta_{w1} = 0$ for all $w\in W_{R-1}$, but by
Lemma~\ref{lem:wcrecursion}, $w1\in W_R$ for all $w\in W_{R-1}$.  The
argument for $g_R(x)$ is similar using $w0$ for $w\in W_{R-1}$.
Therefore, $f_R(x)$ and $g_R(x)$ fully factor into $R-1$ linear terms.

Now, write $v,w$ as the vector of coefficients of $f_R,g_R$,
respectively.  Each of these vectors are of length $2^{R-1}$.  We have
shown that each corresponds to a factorizing polynomial, and we needed
$|W_{R-1}| = 2^{R-1}-R$ conditions to establish this.  Thus $v,w$ each
only have $R$ degrees of freedom.  Therefore, exactly $R-1$ additional
conditions guarantee that $v,w$ are linearly dependent.  Note that
there are $R-1$ sequences of the second type in
Lemma~\ref{lem:wcrecursion}, and moreover, notice that each of them
involve a pair of coefficients that does not appear in any of the
others, e.g. the sequence with its first $\w$ in the $i$th slot
involves the coefficients of $x_ix_R, x_i$.  Thus they are all
independent, and these conditions are sufficient to guarantee the
collinearity of $v$ and $w$.

We have shown that $f_R(x), g_R(x)$ each fully factor into $R-1$
linear terms, and that they are scalar multiples of each other (say
$g_R(x) = \beta f_R(x)$).  Then we can write
\begin{equation*}
  \M(\G) = (x_R+\b) f_R(x),
\end{equation*}
and this is clearly a product of $R$ linear factors.
\end{proof}

Finally we note that the arguments of the previous section tell us how
to interpret $\tilde b$ as a signed count of spanning forests in a
certain derived graph, and this is proved the same way as
Theorem~\ref{thm:2-forest}.

\begin{prop}
Let $\Gamma_{\tilde b}$ be the (multi)graph derived from $\Gamma$ in the 
following way: for each $1$ digit of $\tilde b$ the corresponding edge of 
$\Gamma$ is contracted, while for each $0$ digit of $\tilde b$ the 
corresponding edge of $\Gamma$ is deleted.  Then
\[
\Delta_{\tilde b} =  \left(\sum _{F \in F^{(2)}}(-1)^{n(F)} \right)^2
\]
\end{prop}

\section{Stability Estimates}
 
Here we prove a sufficient condition for stability of a Laplacian in
terms of the ``worst edge''.

\begin{thm}
  Let $\G$ be a weighted graph with fixed positive weights with $\G_+$
  connected and $R$ red edges.  For each red edge $e_i$, let $\Gamma_i
  = \Gamma_{\{i\}}$ be as defined as in Definition~\ref{def:SIJ}.
  Define
  \begin{equation*}
    \omega_i = \frac{\M(\G_i)}{\M(\G_+)}.
  \end{equation*}
  Then, for any ${\bf t} = (t_1,t_2,\dots, t_R)$ with 
  \begin{equation*}
    \norm{{\bf t}}_{\ell^1} \le \min_i \omega_i,
  \end{equation*}
  define $G_{\mathbf t}$ as the graph where we associate weight $-t_i$
  to edge $e_i$ has spectral index.  Then the spectral index of
  $\mathcal{L}(G_{\mathbf t})$ is $(n_-,n_0,n_+) = (N-1,1,0)$.
\end{thm}

\begin{proof}
  For each $i$, let $e_i$ be the standard basis vector with a $1$ in
  slot $i$.  If ${\mathbf t} = \alpha e_i$, then for all $\alpha <
  t_i$, $\mathbf{L}(G_{\mathbf t})$ is negative semidefinite with
  $N-1$ negative eigenvalues.  More generally, if $\norm{{\bf
      t}}_{\ell^1} \le \min_i \omega_i$, then ${\bf t}$ is a convex
  linear combination of the basis vectors described above, and is thus
  a convex linear combination of positive semidefinite operators, and
  thus is itself positive semidefinite.  Moreover, if we restrict to
  the subspace ${\bf 1}^\perp$, then the previous sentence is still
  true after replacing ``positive semidefinite'' with ``positive''.
\end{proof}

\begin{remark}
  Basically, this just tells us that if we can rank the red edges from
  ``best'' to ``worst'', then as long as we would not lose stability
  by putting all of the weight on the worst edge, then we can
  redistribute the same amount of negative weights however we choose
  and still retain stability.
\end{remark}

\section{Numerical Results}

\newcommand{\gp}{G_{\mathsf{P}}} \newcommand{\gm}{G_{\mathsf{M}}}

\begin{figure}[ht]
\begin{centering}
  \includegraphics[width=0.9\textwidth]{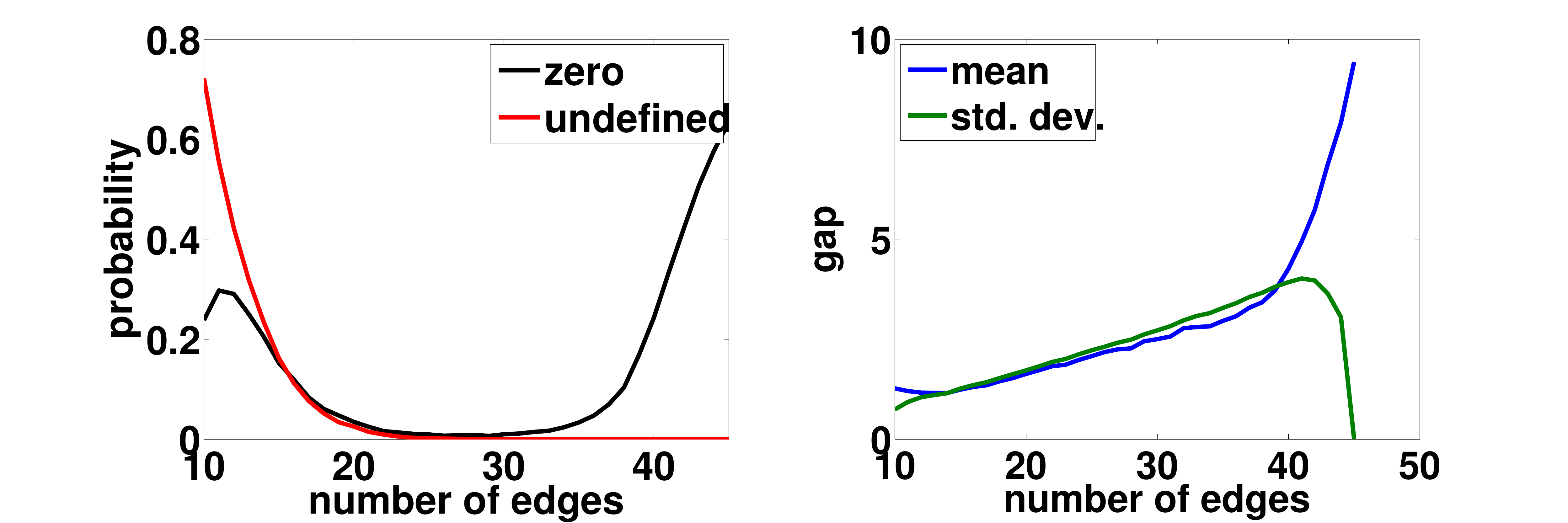}
  \includegraphics[width=0.9\textwidth]{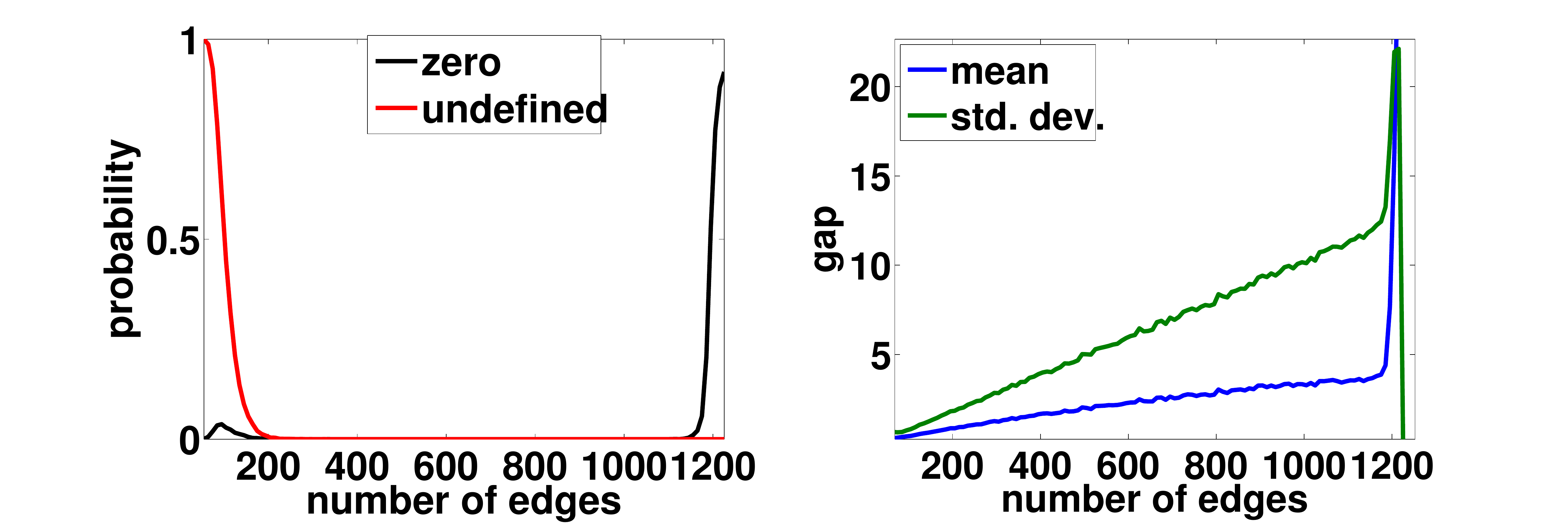}
  \caption{This figure contains several observables for $\gm(N,M)$
    versus $M$, where we choose $N=10$ in the top row and $N=50$ in
    the bottom.  In the left figure, we plot both the probability that
    $\Delta = 0$ (black) and the probability that $\G_+$ is not
    connected (red).  On the right, we plot the mean and standard
    deviation of the gap conditioned on $\G_+$ being connected and
    $\Delta \neq 0$.}
  \label{fig:gm1}
\end{centering}
\end{figure}

\begin{figure}[ht]
\begin{centering}
  \includegraphics[width=0.9\textwidth]{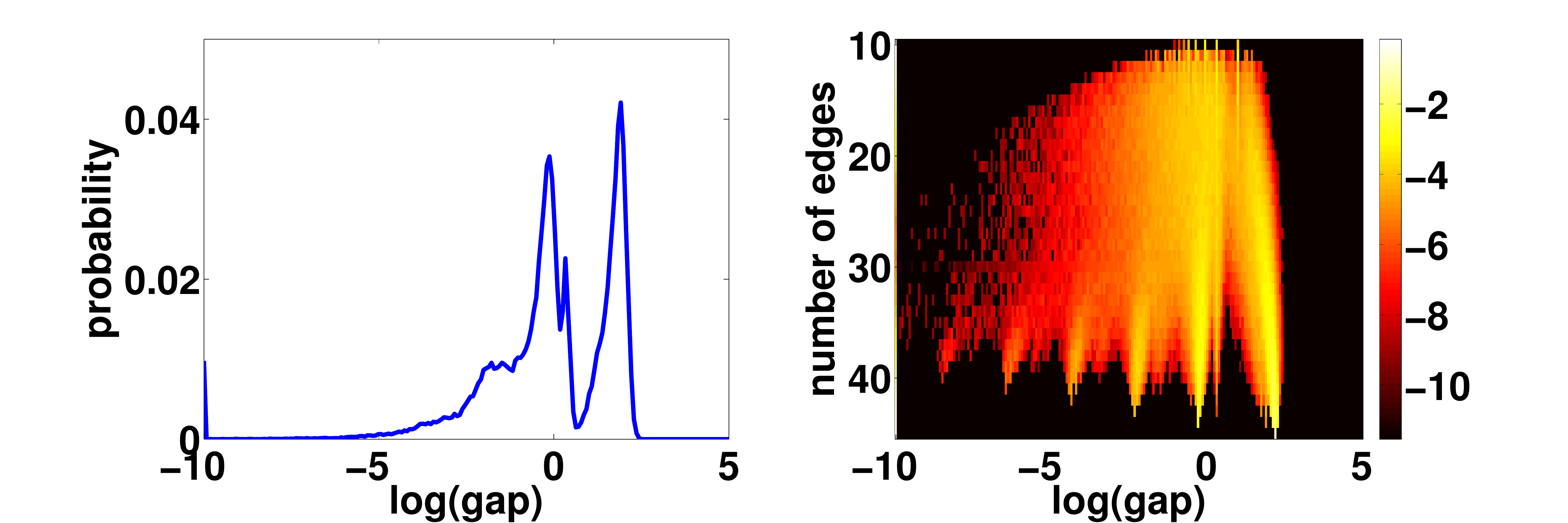}
  \includegraphics[width=0.9\textwidth]{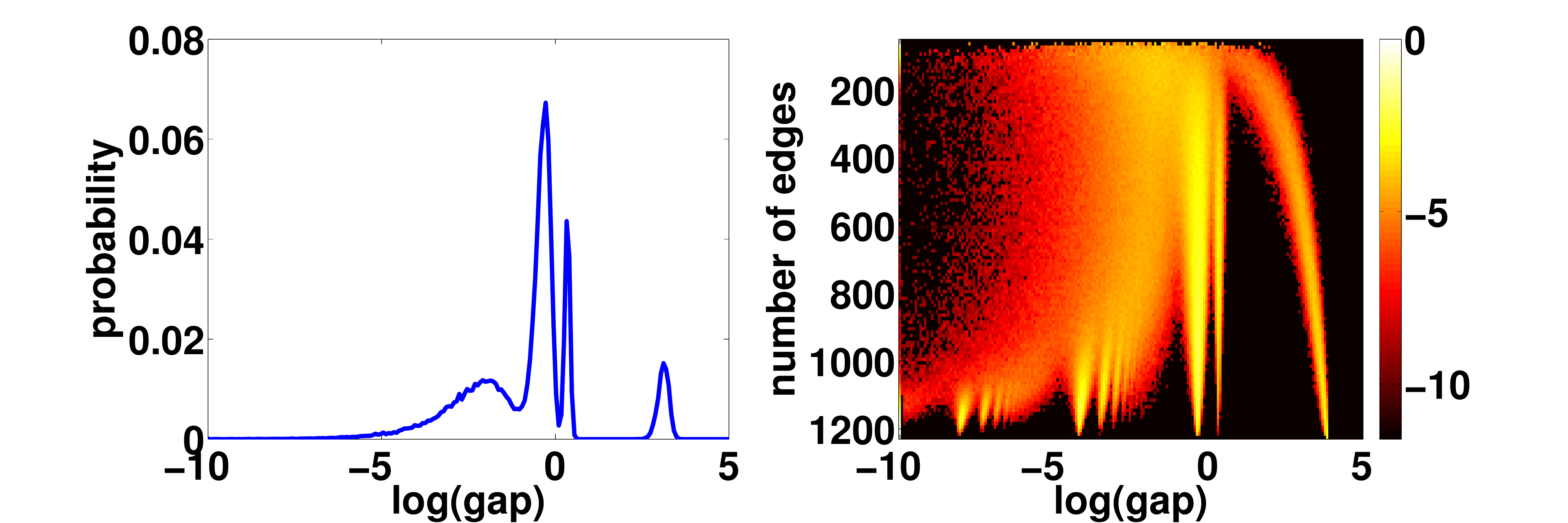}
  \caption{ Histograms of the size of the gap for the $\gm(N,M)$
    model.  }
  \label{fig:gm2}
\end{centering}
\end{figure}

\begin{figure}[ht]
\begin{centering}
  \includegraphics[width=1.0\textwidth]{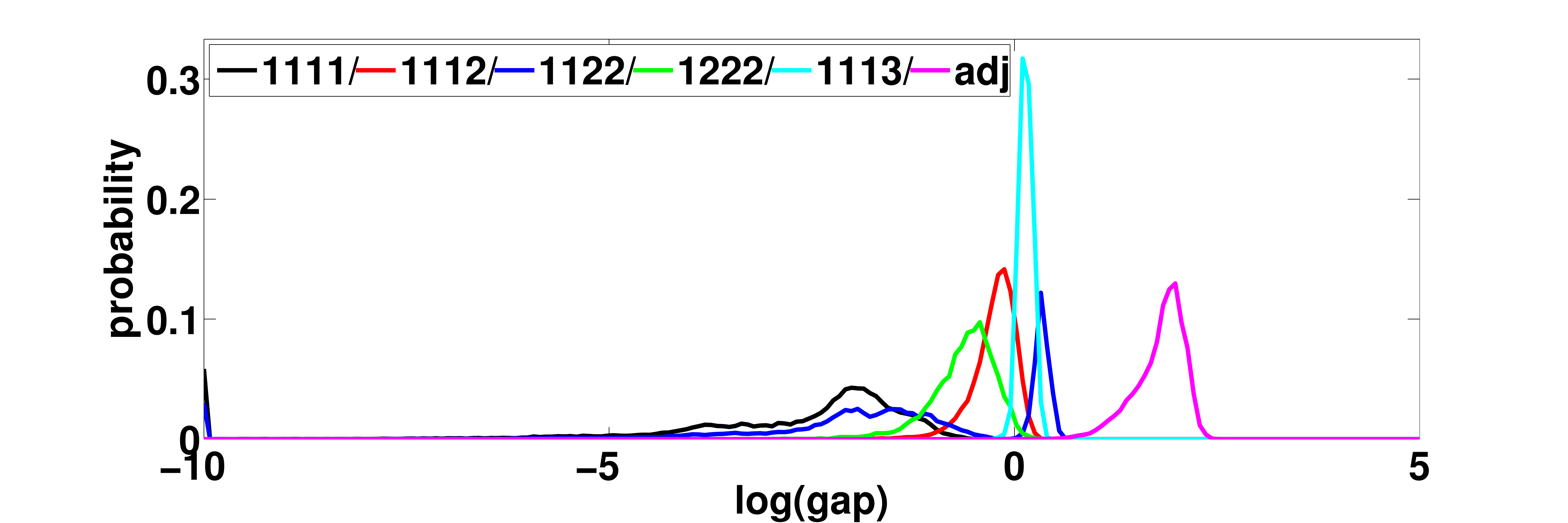}
  \caption{This figure is an overlay of several conditional
    distributions for $N=10,M=30$, with specifics described in the
    text.  In particular, the histogram in the top left frame of
    Figure~\ref{fig:gm2} is a convex linear combination of the various
    histograms presented here.}
  \label{fig:conditional}
\end{centering}
\end{figure}

In this section we present a collection of numerical results for
random matrices.  We are particular interested the probability
distribution of the nondimensionalized ``gap'' of a random matrix with
two red edges.  In all of the simulations performed, we have used the
Erd\H{o}s--R\'{e}nyi random graphs $\gm(N,M)$, the uniform random
graph with $N$ vertices and $M$ edges.  (The specific definition for
the distribution on these graphs is as follows.  Fix $N$, order the
$N(N-1)/2$ edges in some manner, and then choose $M$ of them without
replacement.)  To make this a signed graph, we then uniformly choose
two of these edges to be red.  (We also performed all of the simulations
presented here for the graph model $\gp(N,p)$ (in this case, one
chooses each edge to be present independently with probability $p$)
and the results are almost indistinguishable.  In the interests of
space we do not present these here.)

We see a variety of interesting behaviors in these numerics.  First,
we plot several coarse observables of this ensemble in
Figure~\ref{fig:gm1}.  The first quantity of interest is the
probability that the gap is zero, and the probability that it is
undefined.  Recall that the gap will be undefined whenever the graph
is disconnected, since in this case $A_{11}$ is zero
(q.v.~\eqref{eq:Euc2}).  Of course, this is more likely for a small
number of edges, and we see these curves decreasing monotonically.
Interestingly enough, the probability of a zero gap is actually not
monotone, but in fact turns around: it is most likely when there are
very few edges and when there are many edges.  In fact, we see that
for sparse graphs, the vast majority of the connected graphs have zero
gaps.  In all of these pictures, the number of edges increases until
we reach the complete graph; as proved in the text, the event of
having a zero gap for the complete graph is the same as having the two
red edges not share a vertex (q.v.~Example~\ref{exa:complete}).  A
simple combinatorial argument shows that this probability is $1-4/N$,
giving $0.6, 0.92$ in the two graphs plotted in Figure~\ref{fig:gm1}.

In Figure~\ref{fig:gm2}, we plot the histograms of the gap in a
variety of ways.  As in Figure~\ref{fig:gm1}, the top row corresponds
to $N=10$ and the bottom row to $N=50$.  In the left column, we have
plotted the histogram for a single value of $M$; in the top left, we
have $N=10, M=30$ and in the bottom left we have $N=50, M=605$.  In
the right column, we plot a heatmap of the histograms for various $M$;
in the top right, we plot all $M$ in the range $[10,45]$ and in the
bottom right, we have $M$ going from $55$ to $1225$ by jumps of $10$.
Note that in each row, the left frame is a horizontal ``slice'' of the
right frame.  Each ``slice'' in the right figure contains $10^5$
realizations of random graphs; for the specific histograms in the left
column we simulated $10^7$ random graphs.

What is most striking about this figure is the lack of normality of
the distribution, yet, at the same time, the clear impression that
they are multimode normal.  Moreover, the heatmaps certainly suggest
that these modes are coherent as a function of $M$.  Again, we know
that the distribution for the complete graph is degenerate in that it
can take on only two values: zero (if the edges are disjoint) and
$N^{2N-6}$ (if the edges share a vertex).  Since we are plotting the
logarithm of the gap, the zero is represented by a mass at $-\infty$
(which is here binned at $-10$).  So we see that only one nonzero mode
``survives'' as $M$ goes to the complete graph, and this is the one
corresponding to adjacent edges.

It is natural to ask what these modes correspond to, and we have
decomposed the distribution in Figure~\ref{fig:conditional}.  We have
done this for the case $N=10, M=30$, so that this is a decomposition
of the histogram in the top left frame of Figure~\ref{fig:gm2}.  What
we have done is as follows: for all of the random graphs generated for
$N=10, M=30$, we condition these in various ways.  Those where the red
edges are adjacent we put in class ``adj''.  For those in which the
red graphs were not, we computed four pairwise path distances between
the vertices of the two red edges.  To be more specific, if the two
red edges were $x = (x_1,x_2)$ and $y=(y_1,y_2)$, we computed
$d_{G_+}(x_i,y_j)$, $i,j=1,2$ --- that is to say, we computed the
graph distance between the nodes using only black edges.  If, for
example, all of these vertices were adjacent in the black graph, we
put it in class ``1111''; if, for example, three of them were
adjacent, but one pair had a $G_+$ distance of $2$, we put it in class
``1112'', etc.  We then plotted the conditional distributions for each
class, and we see the conditional distributions are close to
log-normal (recall that the horizontal axis is always the logarithm of
the gap).  

In summary, while the distribution of all graphs with a fixed number
of edges is multimodal, if we consider those graphs with a fixed
number of edges $M$, and two red edges, and we condition on the
geometry of these red edges, then the conditional distributions are
each quite close to log-normal.

\section*{Acknowledgments}

J.B. was supported by the National Science Foundation under grant
DMS-1211364.  L.D. was supported by the National Science Foundation
under grants CMG-0934491 and UBM-1129198 and by the National
Aeronautics and Space Administration under grant NASA-NNA13AA91A.

\end{document}